\documentclass[10pt,a4paper]{article}

\usepackage[T1]{fontenc}
\usepackage[utf8]{inputenc}
\usepackage[margin=1in]{geometry}
\usepackage{amsmath,amssymb,amsthm}
\usepackage{mathtools}
\usepackage{hyperref}

\hypersetup{
  colorlinks=true,
  linkcolor=blue,
  citecolor=blue,
  urlcolor=blue
}

\newtheorem{theorem}{Theorem}[section]
\newtheorem{lemma}[theorem]{Lemma}
\newtheorem{corollary}[theorem]{Corollary}
\newtheorem{claim}[theorem]{Claim}
\theoremstyle{definition}
\newtheorem{definition}[theorem]{Definition}
\theoremstyle{remark}

\newcommand{\bhn}{\widetilde{\alpha}}
\newcommand{\floor}[1]{\left\lfloor #1\right\rfloor}
\newcommand{\ceil}[1]{\left\lceil #1\right\rceil}

\title{\bfseries Edge-disjoint Hamilton cycles under a bipartite-hole condition}
\author{
Yanan Hu\textsuperscript{*1}
\quad
Chengli Li\textsuperscript{$\dagger$2}
\quad
Feng Liu\textsuperscript{$\ddagger$3}\\[1ex]
\small\textsuperscript{1}School of Science, Shanghai Institute of Technology, Shanghai 201418, China\\
\small\textsuperscript{2}Department of Mathematics, East China Normal University, Shanghai 200241, China\\
\small\textsuperscript{3}School of Mathematical Sciences, Shanghai Jiao Tong University, Shanghai 200240, China
}
\date{}

\makeatletter
\renewcommand{\@maketitle}{%
  \newpage
  \null
  \vskip -1.5em%
  \begin{center}%
    {\LARGE \@title \par}%
    \vskip 1.5em%
    {\large \lineskip .5em%
      \begin{tabular}[t]{c}\@author\end{tabular}\par}%
  \end{center}%
  \par
  \vskip 1.5em}
\makeatother

\begin{document}

\maketitle
\begingroup
\renewcommand{\thefootnote}{\fnsymbol{footnote}}
\footnotetext[1]{Email: \texttt{hg@sit.edu.cn}}
\footnotetext[2]{Email: \texttt{lichengli0130@126.com}}
\footnotetext[3]{Email: \texttt{liufeng0609@126.com}}
\endgroup

\begin{abstract}
In 2017, McDiarmid and Yolov introduced the bipartite-hole-number $\widetilde{\alpha}(G)$ and proved that $\delta(G)\ge \widetilde{\alpha}(G)$ forces a Hamilton cycle. They also gave a sufficient condition for packing edge-disjoint Hamilton cycles, and asked whether this condition is sharp or can be relaxed. For integers $a,k\ge 2$, let $f(a,k)$ be the least integer $d$ such that every graph $G$ on at least three vertices with $\widetilde{\alpha}(G)\le a$ and $\delta(G)\ge d$ contains $k$ pairwise edge-disjoint Hamilton cycles. We prove that
$
f(a,k)=\Theta\left(a+k+\frac{ak}{\log(k+2)}\right).
$
The upper bound uses a deletion lemma for the bipartite-hole-number together with the McDiarmid--Yolov Hamiltonicity theorem and a greedy packing argument. The lower bound is obtained from three extremal constructions, the logarithmic one using a sparse random auxiliary graph with no prescribed bipartite hole.
\end{abstract}



\section{Introduction}

Dirac's classical theorem states that every graph with order $n\ge 3$ and minimum degree at least $n/2$ contains a Hamilton cycle \cite{Dirac1952}. Ore's degree-sum theorem \cite{Ore1960} and the Chv\'atal--Erd\H{o}s theorem \cite{ChvatalErdos1972} are two fundamental extensions. In the latter, the minimum-degree condition is replaced by a condition involving connectivity and the independence number.

McDiarmid and Yolov \cite{McDiarmidYolov2017} introduced a parameter which interpolates between independence-type and density-type assumptions. Given positive integers $s,t$, an $(s,t)$-bipartite hole in a graph $G$ is a pair of disjoint vertex sets $S,T\subseteq V(G)$ with $|S|=s$, $|T|=t$, and no edges between $S$ and $T$. Following their notation, we denote the bipartite-hole-number of $G$ by $\bhn(G)$. It is the least integer $r$ that may be written as $r=s+t-1$ for some positive integers $s,t$ such that $G$ contains no $(s,t)$-bipartite hole. Equivalently, $\bhn(G)\ge r$ if and only if, for every positive pair $s,t$ with $s+t\le r$, the graph $G$ contains an $(s,t)$-bipartite hole. In particular, $\bhn(G)\le a$ if and only if there are positive integers $s,t$ with $s+t\le a+1$ such that $G$ contains no $(s,t)$-bipartite hole. Observe that $\bhn(G)=1$ if and only if $G$ is complete.

We denote by $\delta(G)$ and $\Delta(G)$ the minimum degree and maximum degree of $G$, respectively, and by $d_G(v)$ the degree of a vertex $v$. McDiarmid and Yolov proved the following sharp Hamiltonicity criterion.

\begin{theorem}[McDiarmid--Yolov \cite{McDiarmidYolov2017}]
A graph $G$ with at least three vertices is Hamiltonian if $\delta(G)\ge \bhn(G)$.
\end{theorem}

They also obtained a sufficient condition for packing Hamilton cycles.

\begin{theorem}[McDiarmid--Yolov \cite{McDiarmidYolov2017}]
Let $G$ be a graph with at least three vertices, and let $k\ge 1$. If
$$
  \delta(G)\ge k\bhn(G)+3(k-1),
$$
then $G$ contains $k$ pairwise edge-disjoint Hamilton cycles.
\end{theorem}

The case $k=1$ is exactly the Hamiltonicity theorem. For $k\ge2$, however, the minimum-degree threshold is less clear. In the concluding remarks of their paper \cite{McDiarmidYolov2017}, McDiarmid and Yolov wrote that, for the edge-disjoint extension, no sharpness examples were known when the number of cycles is at least two, and that it would be interesting either to find such examples or to relax the condition. This asks whether the condition
$
  \delta(G)\ge k\bhn(G)+3(k-1)
$
is close to the best possible consequence of the two parameters $\bhn(G)$ and $\delta(G)$, or whether it can be substantially improved. This paper gives the order-of-magnitude answer: the best possible threshold is smaller by a logarithmic factor in the mixed term.

For integers $a,k\ge2$, define $f(a,k)$ to be the least integer $d$ such that every graph $G$ on at least three vertices with $\bhn(G)\le a$ and $\delta(G)\ge d$ contains $k$ pairwise edge-disjoint Hamilton cycles. Our main result is the following. All logarithms are natural.

\begin{theorem}\label{thm:main}
There are absolute constants $c,C>0$ such that, for every pair of integers $a,k\ge2$,
$$
 c\left(a+k+\frac{ak}{\log(k+2)}\right)
 \le f(a,k) \le
 C\left(a+k+\frac{ak}{\log(k+2)}\right).
$$
\end{theorem}

Thus the answer differs from the direct McDiarmid--Yolov sufficient condition by a logarithmic factor in the mixed term. The proof also gives a more explicit sufficient condition. We first introduce the auxiliary function used in the upper bound.

\begin{definition}
For integers $a\ge1$ and $D\ge0$, define $\Phi(a,D)$ as follows. Put $\Phi(a,0)=a$, and, for $D\ge1$, put
$$
\Phi(a,D)=
\max_{1\le s\le \floor{(a+1)/2}}
\min_{\substack{x\in \mathbb N\\ x\ge 2s+2D}}
\left(
 x+
 \ceil{(a+1-s)\exp\left(\frac{4sD}{x}\right)}-1
\right).
$$
\end{definition}

\begin{theorem}\label{thm:explicit-upper}
Let $G$ be a graph on at least three vertices. Let $k\ge1$, and let $a\ge1$ be an integer with $a\ge \bhn(G)$. If
$$
 \delta(G)\ge \max_{0\le i\le k-1}\bigl(2i+\Phi(a,2i)\bigr),
$$
then $G$ contains $k$ pairwise edge-disjoint Hamilton cycles.
\end{theorem}

The upper bound in Theorem \ref{thm:main} follows from Theorem \ref{thm:explicit-upper} and the estimate
$$
 \Phi(a,D)\le C_0\left(a+D+\frac{aD}{\log(D+2)}\right),
$$
where $C_0$ is an absolute constant. The matching lower bound, up to absolute constants, is proved in Section~\ref{sec:lower}. Related developments around bipartite holes and Hamiltonian properties include work of Chen \cite{Chen2022}, Zhou, Broersma, Wang and Lu \cite{ZhouBroersmaWangLu2024}, Dragani\'c, Munh\'a Correia and Sudakov \cite{DraganicCorreiaSudakov2024}, Ellingham, Huang and Wei \cite{EllinghamHuangWei2025}, Li and Liu \cite{LiLiu2025}, Li, Liu and Tang \cite{LiLiuTang2025}, and Cheng and Tang \cite{ChengTang2025}. The general problem of packing edge-disjoint Hamilton cycles has a long history, going back at least to Nash-Williams \cite{NashWilliams1971}; see also Christofides, K\"uhn and Osthus \cite{ChristofidesKuhnOsthus2012} and K\"uhn and Osthus \cite{KuhnOsthus2014} for dense-graph packing and decomposition results, and Knox, K\"uhn and Osthus \cite{KnoxKuhnOsthus2015} and Krivelevich and Samotij \cite{KrivelevichSamotij2012} for random-graph packing results. We use only elementary probabilistic tools; standard references are \cite{AlonSpencer2016,Bollobas2001}.

\section{Preliminaries}

We consider finite simple graphs. For a graph $G$ and disjoint vertex sets $X,Y\subseteq V(G)$, let $E_G(X,Y)$ denote the set of edges with one endpoint in $X$ and the other in $Y$. We denote by $G[X,Y]$ the bipartite graph induced by the edges of $G$ between $X$ and $Y$. A spanning subgraph $F$ of $G$ is a subgraph with vertex set $V(G)$. The notation $G-F$ means that the edges of $F$ are deleted from $G$, while all vertices are kept.

The next lemma is the basic deletion lemma. It says that deleting a bounded-maximum-degree spanning subgraph increases the bipartite-hole-number by at most a controlled logarithmic factor beyond the trivial obstruction.

\begin{lemma}\label{lem:deletion}
Let $G$ be a graph containing no $(s,t)$-bipartite hole, where $s,t\ge1$. Let $D$ be a nonnegative integer, and let $F$ be a spanning subgraph of $G$ with $\Delta(F)\le D$. If $x,y$ are positive integers satisfying
$$
 x\ge 2s+2D
 \qquad\text{and}\qquad
 y\ge t\exp\left(\frac{4sD}{x}\right),
$$
then $G-F$ contains no $(x,y)$-bipartite hole.
\end{lemma}

\begin{proof}
Suppose, to the contrary, that $G-F$ contains an $(x,y)$-bipartite hole $(X,Y)$. Thus $|X|=x$, $|Y|=y$, $X\cap Y=\emptyset$, and $E_{G-F}(X,Y)=\emptyset$. Consequently every edge of $G$ between $X$ and $Y$ belongs to $F$. Let $B=G[X,Y]$. Then $B$ is a bipartite graph with parts $X,Y$ and maximum degree at most $D$.

Choose uniformly at random an $s$-element set $S\subseteq X$. Fix a vertex $v\in Y$, and write $d=d_B(v)\le D$. The assumption $x\ge 2s+2D$ implies $x-d\ge s$, so the binomial coefficient $\binom{x-d}{s}$ is nonzero. The following calculation is therefore valid:
$$
 \Pr\bigl(N_B(v)\cap S=\emptyset\bigr)
 =\frac{\binom{x-d}{s}}{\binom{x}{s}}
 =\prod_{j=0}^{s-1}\left(1-\frac{d}{x-j}\right).
$$
For $0\le u\le 1/2$ we have $1-u\ge \exp(-2u)$. Moreover, for every $0\le j\le s-1$,
$$
 \frac{d}{x-j}
 \le \frac{D}{x-s+1}
 \le \frac12,
 \qquad
 x-s+1\ge \frac{x}{2}.
$$
Hence
$$
 \Pr\bigl(N_B(v)\cap S=\emptyset\bigr)
 \ge
 \exp\left(-2\sum_{j=0}^{s-1}\frac{d}{x-j}\right)
 \ge
 \exp\left(-\frac{4sD}{x}\right).
$$
It follows that the expected number of vertices in $Y$ with no neighbour in $S$ in the graph $B$ is at least
$$
 y\exp\left(-\frac{4sD}{x}\right)\ge t.
$$
Therefore there exists an $s$-set $S\subseteq X$ for which at least $t$ vertices of $Y$ have no neighbour in $S$ in $B$. Choose $T\subseteq Y$ of size $t$ among those vertices. Since $B=G[X,Y]$, we have $E_G(S,T)=\emptyset$, contradicting the assumption that $G$ has no $(s,t)$-bipartite hole.
\end{proof}

\begin{lemma}\label{lem:Phi}
For every integer $a\ge1$, every graph $G$ with $\bhn(G)\le a$, and every spanning subgraph $F$ of $G$ with $\Delta(F)\le D$,
$$
 \bhn(G-F)\le \Phi(a,D).
$$
\end{lemma}

\begin{proof}
The case $D=0$ is immediate. Assume $D\ge1$. Since $\bhn(G)\le a$, there are positive integers $s,t$ with $s+t-1\le a$ such that $G$ has no $(s,t)$-bipartite hole. Indeed, otherwise $G$ would contain an $(s,t)$-bipartite hole for every positive pair $s,t$ with $s+t\le a+1$, which would imply $\bhn(G)\ge a+1$. By symmetry, we may assume $s\le t$. Then
$$
 1\le s\le \floor{\frac{a+1}{2}},
 \qquad
 t\le a+1-s.
$$
Fix an integer $x\ge 2s+2D$, and put
$$
 y=\ceil{(a+1-s)\exp\left(\frac{4sD}{x}\right)}.
$$
Then $y\ge t\exp(4sD/x)$, so Lemma \ref{lem:deletion} implies that $G-F$ has no $(x,y)$-bipartite hole. Hence $\bhn(G-F)\le x+y-1$. Minimising over all admissible $x$ and then taking the worst possible value of $s$ gives the claimed bound.
\end{proof}

\section{The upper bound}

We first prove the explicit version, Theorem \ref{thm:explicit-upper}.

\begin{proof}[Proof of Theorem \ref{thm:explicit-upper}]
We construct the Hamilton cycles greedily. Suppose that $H_1,\ldots,H_i$ have already been chosen, where $0\le i\le k-1$, and that they are pairwise edge-disjoint Hamilton cycles of $G$. Let
$$
 F_i=H_1\cup\cdots\cup H_i,
 \qquad
 G_i=G-F_i.
$$
Since \(H_1,\ldots,H_i\) are pairwise edge-disjoint Hamilton cycles, 
their union \(F_i\) is a spanning \(2i\)-regular subgraph of \(G\). Hence
\[
  \Delta(F_i)=2i,\qquad \delta(G_i)=\delta(G)-2i.
\]
By Lemma \ref{lem:Phi},
$$
 \bhn(G_i)\le \Phi(a,2i).
$$
The displayed hypothesis of the theorem gives
$$
 \delta(G_i)=\delta(G)-2i\ge \Phi(a,2i)\ge \bhn(G_i).
$$
The McDiarmid--Yolov Hamiltonicity theorem therefore implies that $G_i$ contains a Hamilton cycle. Choose such a cycle as $H_{i+1}$. Since $H_{i+1}\subseteq G_i$, it is edge-disjoint from $H_1,\ldots,H_i$. Iterating for $i=0,1,\ldots,k-1$ gives the desired $k$ Hamilton cycles.
\end{proof}

It remains to estimate $\Phi$.

\begin{lemma}\label{lem:Phi-estimate}
There is an absolute constant $C_0$ such that, for all integers $a\ge1$ and $D\ge0$,
$$
 \Phi(a,D)\le C_0\left(a+D+\frac{aD}{\log(D+2)}\right).
$$
\end{lemma}

\begin{proof}
The case $D=0$ is trivial. Assume $D\ge1$, and put $\ell=\log(D+2)$. Fix an integer $s$ with $1\le s\le \floor{(a+1)/2}$. Choose
$$
 x=\ceil{4s+4D+\frac{16sD}{\ell}}.
$$
Then $x\ge 2s+2D$. Moreover,
$$
 \frac{4sD}{x}\le \frac{\ell}{4},
$$
and so
$$
 \exp\left(\frac{4sD}{x}\right)\le (D+2)^{1/4}.
$$
Since $(D+2)^{1/4}\le C_1(1+D/\ell)$ for an absolute constant $C_1$ and all $D\ge1$ (enlarge $C_1$ for bounded $D$, while for large $D$ we have $(D+2)^{1/4}\le D/\log(D+2)$), we have
$$
\begin{aligned}
 x+\ceil{(a+1-s)\exp\left(\frac{4sD}{x}\right)}-1
 &\le
 C_2\left(s+D+\frac{sD}{\ell}+a+\frac{aD}{\ell}\right)  \\
 &\le
 C_3\left(a+D+\frac{aD}{\ell}\right).
\end{aligned}
$$
The bound is independent of $s$, and hence it also bounds the maximum in the definition of $\Phi(a,D)$.
\end{proof}

\begin{corollary}\label{cor:upper}
There is an absolute constant $C>0$ such that, for all integers $a\ge1$ and $k\ge1$, every graph $G$ on at least three vertices with $\bhn(G)\le a$ and
$$
 \delta(G)\ge C\left(a+k+\frac{ak}{\log(k+2)}\right)
$$
contains $k$ pairwise edge-disjoint Hamilton cycles.
\end{corollary}

\begin{proof}
Put
$$
 M=a+k+\frac{ak}{\log(k+2)}.
$$
By Theorem \ref{thm:explicit-upper}, it suffices to prove that
$$
 2i+\Phi(a,2i)\le C'M
$$
for every $0\le i\le k-1$, where $C'$ is an absolute constant.

Let $D=2i$. Then $0\le D\le 2k$. If $D=0$, then
$$
 2i+\Phi(a,2i)=\Phi(a,0)=a\le M.
$$
Assume now that $D\ge1$. By Lemma \ref{lem:Phi-estimate},
$$
 \Phi(a,D)\le C_0\left(a+D+\frac{aD}{\log(D+2)}\right).
$$
Hence
$$
 2i+\Phi(a,2i)=D+\Phi(a,D)
 \le C_0a+(1+C_0)D+C_0\frac{aD}{\log(D+2)}.
$$
Since $D\le2k$ and the function $x/\log(x+2)$ is increasing on $[0,\infty)$,
$$
 D\le 2k,
 \qquad
 \frac{D}{\log(D+2)}
 \le \frac{2k}{\log(2k+2)}
 \le \frac{2k}{\log(k+2)}.
$$
Therefore
$$
 2i+\Phi(a,2i)
 \le C_0a+2(1+C_0)k+2C_0\frac{ak}{\log(k+2)}
 \le C'M
$$
for some absolute constant $C'$. Taking $C\ge C'$ in the statement, Theorem \ref{thm:explicit-upper} applies.
\end{proof}

\section{The lower bound}\label{sec:lower}

We prove three independent obstructions. Together they show that the upper bound in Corollary \ref{cor:upper} is best possible up to an absolute multiplicative constant.

\begin{lemma}\label{lem:a-and-k-lower}
For $a\ge2$ and $k\ge 2$,
$$
 f(a,k)\ge \max\{a,2k\}.
$$
\end{lemma}

\begin{proof}
We prove the two inequalities separately by constructing counterexamples.

First let $G=K_{a-1,a}$. Then $G$ is not Hamiltonian because its two bipartition classes have different sizes, and $\delta(G)=a-1$. Also $\bhn(G)=a$: every positive pair $s,t$ with $s+t\le a$ can be realised as a bipartite hole inside the larger bipartition class, while no $(1,a)$-bipartite hole exists. Thus minimum degree $a-1$ does not even force one Hamilton cycle, and so $f(a,k)\ge a$.

Next let $G$ be the split graph with vertex partition $V(G)=I\cup Q$, where $|I|=a$ and $|Q|=2k-1$. Put no edges inside $I$, put all edges inside $Q$, and join every vertex of $I$ to every vertex of $Q$. Then $\delta(G)=2k-1$. The independent set $I$ shows $\bhn(G)\ge a$, while the absence of a $(1,a)$-bipartite hole shows $\bhn(G)\le a$. Hence $\bhn(G)=a$. If $G$ contained $k$ pairwise edge-disjoint Hamilton cycles, then every vertex would be incident with $2k$ distinct edges belonging to their union. This is impossible at a vertex of $I$, whose degree is only $2k-1$. Therefore $f(a,k)\ge2k$.
\end{proof}

The remaining construction is responsible for the logarithmic term. We present the random auxiliary graph explicitly, since this is the only probabilistic part of the proof.

\begin{lemma}\label{lem:aklog-lower}
There is an absolute constant $c_0>0$ such that, for all integers $a,k\ge2$,
$$
 f(a,k)\ge c_0\frac{ak}{\log(k+2)}.
$$
\end{lemma}

\begin{proof}
Let
$$
 b=\floor{\frac{a+1}{2}},
 \qquad
 h=\ceil{\frac{a+1}{2}},
$$
so that $b+h=a+1$ and $bh\ge a^2/4$. Choose a sufficiently large absolute constant $C_R$ so that the union-bound estimate below is at most $1/4$ for all $a,k\ge2$. After $C_R$ is fixed, choose a sufficiently small absolute constant $\gamma>0$ such that
$$
 \gamma C_R\le \frac1{100}
 \qquad\text{and}\qquad
 \frac{\gamma}{\log 4}\le 1.
$$
Put
$$
 M=\frac{ak}{\log(k+2)},
 \qquad
 N=\floor{\gamma M}.
$$
If $N<8a$, then
$$
 M<\frac{8a+1}{\gamma}\le \frac{9a}{\gamma},
$$
because $a\ge2$. Lemma \ref{lem:a-and-k-lower} then gives $f(a,k)\ge a\ge (\gamma/9)M$. After decreasing $c_0$, this gives the desired lower bound. We may therefore assume from now on that $N\ge8a$.

We first construct a graph $R$ on a vertex set $B$ of size $N$ with the following two properties: $R$ has no $(b,h)$-bipartite hole, and
$$
 e(R)<k\left(N-a-\floor{\frac N4}\right).
$$

\begin{claim}\label{claim:R}
There exists a graph $R$ on $N$ vertices with these two properties.
\end{claim}

\begin{proof}[Proof of Claim~\ref{claim:R}]
Let
$$
 p=\min\left\{1, C_R\frac{\log(k+2)}{a}\right\},
$$
and choose $R$ from the binomial random graph $G(N,p)$; that is, each of the $\binom{N}{2}$ possible edges is included independently with probability $p$.

We first estimate the probability that \(R\) contains a \((b,h)\)-bipartite hole. If $p=1$, then $R$ is complete and this is immediate. Assume $p<1$. Fix disjoint sets $X,Y\subseteq B$ with $|X|=b$ and $|Y|=h$. There are $bh$ possible edges between $X$ and $Y$. Since the choices of these edges are independent,
$$
 \Pr\bigl(E_R(X,Y)=\emptyset\bigr)
 = (1-p)^{bh}
 \le \exp(-pbh).
$$
Using $p=C_R\log(k+2)/a$ and $bh\ge a^2/4$, we get
$$
 \Pr\bigl(E_R(X,Y)=\emptyset\bigr)
 \le
 \exp\left(-\frac{C_R}{4}a\log(k+2)\right).
$$
The number of possible ordered pairs $(X,Y)$ is at most
$$
 \binom Nb\binom Nh
 \le
 \left(\frac{3eN}{a}\right)^{a+1}
 \le (3ek)^{a+1}.
$$
Indeed, the last inequality follows from $N/a\le \gamma k/\log(k+2)\le k$, using $k\ge2$ and the choice of $\gamma$. Combining the last two displays, the union bound gives
$$
 \Pr\bigl(\text{some $(b,h)$-bipartite hole exists in $R$}\bigr)
 \le
 (3ek)^{a+1}
 \exp\left(-\frac{C_R}{4}a\log(k+2)\right).
$$
By the choice of $C_R$, the last expression is at most $1/4$. Thus
$$
 \Pr\bigl(R\text{ has no $(b,h)$-bipartite hole}\bigr)
 \ge \frac34.
$$

We next control the number of edges of $R$. The choice of $\gamma$ gives
$$
 pN\le \frac{k}{100}.
$$
To see this, if $p<1$, then
$$
 pN\le
 C_R\frac{\log(k+2)}{a}\cdot \gamma\frac{ak}{\log(k+2)}
 =\gamma C_R k
 \le \frac{k}{100}.
$$
If $p=1$, then $a\le C_R\log(k+2)$, and hence
$$
 pN=N\le \gamma\frac{ak}{\log(k+2)}\le \gamma C_R k\le \frac{k}{100}.
$$
Therefore
$$
 \mathbb E e(R)
 =p\binom N2
 \le \frac{pN^2}{2}
 \le \frac{kN}{200}.
$$
We now apply Markov's inequality in its elementary form: if $X$ is a nonnegative random variable and $\lambda>0$, then
$$
 \Pr(X\ge\lambda)\le \frac{\mathbb E X}{\lambda}.
$$
Taking $X=e(R)$ and $\lambda=kN/20$, and using the preceding expectation bound, we obtain
$$
 \Pr\left(e(R)\ge \frac{kN}{20}\right)
 \le
 \frac{kN/200}{kN/20}
 =\frac1{10}.
$$
Thus
$$
 \Pr\left(e(R)<\frac{kN}{20}\right)
 \ge \frac9{10}.
$$
Since $N\ge8a$,
$$
 N-a-\floor{\frac N4}
 \ge \frac N2,
$$
and consequently
$$
 \frac{kN}{20}<k\left(N-a-\floor{\frac N4}\right).
$$
Thus the event $e(R)<kN/20$ implies the desired edge bound.

Finally, the probability that both required properties hold is at least
$$
 1-\frac14-\frac1{10}=\frac{13}{20}>0.
$$
Hence a deterministic graph $R$ satisfying the two required properties exists.
\end{proof}

Fix a graph $R$ as in Claim \ref{claim:R}. We now build the required counterexample $G$. Let
$$
 V(G)=I\cup A\cup B,
 \qquad
 |I|=a,
 \qquad
 |A|=\floor{\frac N4},
 \qquad
 |B|=N.
$$
Put no edges inside $I$, put all possible edges inside $A$, put the graph $R$ inside $B$, and make the three parts $I,A,B$ pairwise complete to each other.

We first verify the bipartite-hole-number. Since $I$ is an independent set of size $a$, every positive pair $s,t$ with $s+t\le a$ occurs as a bipartite hole inside $I$. Hence $\bhn(G)\ge a$. On the other hand, $b+h=a+1$, and $G$ has no $(b,h)$-bipartite hole. Indeed, the union of the two sides of any bipartite hole with both sides nonempty must be contained in a single part, because every two distinct parts are completely joined. It cannot be contained in $I$, since $|I|=a<a+1$; it cannot be contained in $A$, since $A$ is a clique; and it cannot be contained in $B$, since $R$ has no $(b,h)$-bipartite hole. Therefore $\bhn(G)\le a$, and so $\bhn(G)=a$.

Next we show that $G$ cannot contain $k$ pairwise edge-disjoint Hamilton cycles. Let $O=I\cup A$. Consider any Hamilton cycle $C$ of $G$. Let $x(C)$ be the number of edges of $C$ with both endpoints in $B$, let $y(C)$ be the number of edges of $C$ with both endpoints in $O$, and let $z(C)$ be the number of edges of $C$ crossing the cut $(B,O)$. Counting incidences with the vertices of $B$ and $O$ along the cycle $C$, respectively, gives
$$
 2|B|=2x(C)+z(C),
 \qquad
 2|O|=2y(C)+z(C).
$$
Subtracting these equations gives
$$
 x(C)-y(C)=|B|-|O|.
$$
Since $y(C)\ge0$, every Hamilton cycle $C$ satisfies
$$
 x(C)
 \ge |B|-|O|
 =N-a-\floor{\frac N4}.
$$
All edges counted by $x(C)$ are edges inside $B$, and the only edges inside $B$ are the edges of $R$. Therefore any family of $k$ pairwise edge-disjoint Hamilton cycles would use at least
$$
 k\left(N-a-\floor{\frac N4}\right)
$$
distinct edges of $R$. This is impossible by the edge bound on $R$. Hence $G$ contains no $k$ pairwise edge-disjoint Hamilton cycles.

Finally, we estimate the minimum degree. Vertices in $B$ have degree at least $a+\floor{N/4}$, vertices in $I$ have degree $N+\floor{N/4}$, and vertices in $A$ have degree $a+N+\floor{N/4}-1$. Thus
$$
 \delta(G)\ge a+\floor{\frac N4}.
$$
Since $N\ge8a\ge16$ and $N=\floor{\gamma M}$, we have $N\ge \gamma M/2$, and therefore
$$
 \delta(G)\ge \floor{\frac N4}\ge \frac N8\ge \frac{\gamma}{16}M.
$$
After decreasing $c_0$ if necessary, the constructed graph satisfies $\bhn(G)=a$, has minimum degree at least $c_0ak/\log(k+2)$, and contains no $k$ pairwise edge-disjoint Hamilton cycles. This proves the lemma.
\end{proof}

\begin{proof}[Proof of Theorem \ref{thm:main}]
The upper bound is Corollary \ref{cor:upper}. For the lower bound, Lemmas \ref{lem:a-and-k-lower} and \ref{lem:aklog-lower} give
$$
 f(a,k)\ge
 \max\left\{a,2k,c_0\frac{ak}{\log(k+2)}\right\}.
$$
The maximum of three nonnegative numbers is at least one third of their sum. After adjusting the absolute constant, this gives
$$
 f(a,k)\ge
 c\left(a+k+\frac{ak}{\log(k+2)}\right),
$$
as required.
\end{proof}

\section{Concluding remarks}

Theorem \ref{thm:main} determines, up to absolute multiplicative constants, the minimum-degree threshold for forcing $k$ pairwise edge-disjoint Hamilton cycles when the hypothesis may depend only on $\bhn(G)$ and $\delta(G)$. In particular, it answers the problem raised in the concluding remarks of McDiarmid and Yolov's 2017 paper \cite{McDiarmidYolov2017} at the level of order of magnitude: the packing condition can be relaxed from the mixed term $ak$ to $ak/\log k$, and this logarithmic improvement is best possible up to constants.

The proof does not attempt to optimise constants. It would be interesting to understand the best constant in front of the term $ak/\log k$, or to determine whether additional information about the profile of forbidden bipartite holes leads to sharper non-uniform packing criteria.

\section*{Acknowledgments}
This work was supported by the Science and Technology Commission of Shanghai Municipality (No. 25ZR1402474).

\section*{Declaration on the use of AI}
The authors used generative AI tools to assist in discussing proof strategies,
checking proofs, and improving exposition.

\end{document}